\documentclass[12pt,reqno]{amsart}
\usepackage{graphicx}
\usepackage{color}
\usepackage{mathrsfs,amsmath,amssymb,amsthm,amsfonts}

\usepackage{graphics,url,color,epsfig}
\usepackage{tikz}
\usepackage[colorlinks]{hyperref}
\AtBeginDocument{%
  \hypersetup{%
    linkcolor=blue,%
    citecolor=green,%
  }%
}
\usepackage{cleveref}
\usepackage{pgfplots}
\usepackage{multicol}

\def\R{\mathbb{R}}

\makeatletter

\newcommand{\Rmnum}[1]{\expandafter\@slowromancap\romannumeral #1@}
\makeatother

\newtheorem{thm}{Theorem}[section]

\newtheorem{lemma}[thm]{Lemma}

\newtheorem{theorem}[thm]{Theorem}

\usepackage[numbers,sort&compress]{natbib}
\newcommand{\x}{{\bf x}}

\newcommand{\fr}{\displaystyle\frac}

\newcommand{\mb}{\mbox}
\everymath{\displaystyle}

\allowdisplaybreaks
\pgfplotsset{compat=1.18}
\begin{document}

\author{Yahong Guo}
\address{Yahong Guo \newline\indent School of Mathematical Sciences \newline\indent  Shanghai Jiao Tong University \newline\indent Shanghai, 200240, P. R. China}
\email{yhguo@sjtu.edu.cn}

\author{Chilin Zhang}
\address{Chilin Zhang
\newline\indent
School of Mathematics \newline\indent Fudan University
\newline\indent Shanghai 200433, P. R. China}
\email{zhangchilin@fudan.edu.cn}

\title[A Liouville theorem]{Classification of solutions to a weighted singular fractional problem  in the half space}

\begin{abstract}
We focus on the classification of positive solutions to  $(-\Delta)^s u=\frac{x_n^{\alpha}}{u^\gamma}$ in the half space with $\gamma>0$, subject to the Dirichlet condition.  We show that when $-2s<\alpha<(\gamma-1)s$,   all positive solutions  exhibit one-dimensional symmetry and are monotone increasing in $x_n$.  
Moreover, we provide a complete classification of all such one-dimensional solutions via their 
``asymptotic $s$-order slope".
When $\alpha$  lies outside this range,
we demonstrate the nonexistence of global positive solutions.
\end{abstract}


\maketitle

\numberwithin{equation}{section}
\section{Introduction and main results}
\subsection{Background}
We study the classification of positive solutions to the following fractional singular Lane-Emden-Fowler problem 
\begin{equation}\label{eq. main}
    \begin{cases}
    (-\Delta)^{s}u(x)=\frac{x_{n}^{\alpha}}{u^{\gamma}(x)},&x\in\mathbb{R}^{n}_{+},\\
    u(x)>0,&x\in\mathbb{R}^{n}_{+},\\
    u(x)=0&x\in\mathbb{R}^n\backslash\mathbb{R}^{n}_{+},
      \end{cases}
\end{equation}
where $0<s<1$, $n\geq 1$,  $\alpha\in\mathbb{R},$ and $\gamma >0$. In this context, for $x\in\mathbb{R}^n$, we denote  $x=(x',x_{n})$ with $x'\in\mathbb{R}^{n-1}$ and $\mathbb{R}^{n}_{+}=\{x\in\mathbb{R}^{n}:x_n>0\}$.

In \eqref{eq. main}, the symbol $(-\Delta)^s$ denotes the fractional Laplacian, which can be expressed as
\begin{equation}\label{eq1-1}
(-\Delta)^s u(x)=C_{n,s}PV\int_{\mathbb{R}^{n}}\frac{u(x)-u(y)}{|x-y|^{n+2s}}dy=C_{n,s}\lim_{\varepsilon\to 0}\int_{\mathbb{R}^n\backslash B_{\varepsilon}(0)}\frac{u(x)-u(y)}{|x-y|^{n+2s}}dy,
\end{equation}
In order that the integral in \eqref{eq1-1} is well defined, we require that $u\in C_{loc}^{1, 1}(\mathbb{R}^n_+)\cap \mathcal{L}_{2s}\cap C(\mathbb{R}^n)$, where
\begin{equation*}
\mathcal{L}_{2s}=\left\{u\in L_{loc}^1(\mathbb{R}^n) \,\Big| \int_{\mathbb{R}^n}\frac{|u(x)|}{1+|x|^{n+2s}}dx<\infty \right\}
\end{equation*}
endowed naturally with the norm
\begin{equation*}
\|u\|_{\mathcal{L}_{2s}}:= \int_{\mathbb{R}^n}\frac{|u(x)|}{1+|x|^{n+2s}}dx.
\end{equation*}
Because of the non-locality of  the fractional Laplacian, it is necessary to assume $u(x)=0$ in the whole complement of $\mathbb{R}^{n}_{+}$, rather than merely on the boundary $\partial\mathbb{R}^{n}_{+}$.

If $s=1$, $(-\Delta)^s$ becomes the regular Laplacian $-\Delta$.   For functions  $u\in C^2(\mathbb{R}^{n}_{+})\cap C(\overline{\mathbb{R}^{n}_{+}})$, the problem \eqref{eq. main} with $\alpha=0$ simplifies to the following singular local problem:
\begin{equation}\label{eq. main-i}
    \begin{cases}
    -\Delta u(x)=\frac{1}{u^\gamma(x)},&x\in\mathbb{R}^{n}_{+},\\
    u(x)>0,&x\in\mathbb{R}^{n}_{+},\\
     u(x)=0&x\in\partial\mathbb{R}^{n}_{+},
      \end{cases}
\end{equation}
the only boundary condition required is that $u(x)=0$ on $\partial\mathbb{R}^{n}_{+}$, which occurs spontaneously.

From a physical point of view, equations with a singular right-hand side represent a generalized version of the Lane-Emden-Fowler equation and find applications across multiple scientific domains. For instance, in the study of thermo-conductivity \cite{FuMa} where $u^\gamma$ models material resistivity, in the theory of gaseous dynamics in astrophysics \cite{Fo}, in signal transmission \cite{No}, and in the context of chemical heterogeneous catalysts \cite{Pe}. Among these applications, one of the most physically significant implementations occurs in the study of non-Newtonian pseudo-plastic fluids, where such singular equations model boundary layer phenomena (see e.g.,\cite{AcShpe,Naca,VasoMo} and references therein).

The study of semi-linear elliptic problems involving singular nonlinearities in bounded domains has attracted significant attention, beginning  with the seminal work in \cite{CRT1977}. For further developments, see also \cite{BO2010, BCT2014, delP1992, LM1991, OP2018}. For nonlocal singular problems in bounded domains, the readers may refer to \cite{AGJ2018, CMSS2017, GMS2017} and the related studies.

For $\gamma>0$, the singular problem $-\Delta u=u^{-\gamma}$ in half spaces, which arises from general equations $-\Delta u=u^{-\gamma}+f(u)$, often appears in blow-up analysis near the boundary of a domain. This is because the term $u^{-\gamma}$ becomes dominant as $u$ approaches zero. Such scenarios are discussed in \cite{CES2019,JFA2020}, where the authors examine variations of the Hopf lemma. Recently, Montoro-Muglia-Sciunzi in \cite{MMS,MMS1} provided a classification result for positive solutions to \eqref{eq. main-i} in the half space by establishing both the upper and lower bound estimates.

For the fractional case, it was shown in \cite{GuWu25} that global solutions to \eqref{eq. main} with $\alpha=0$ are all one-dimensional. The key difference between \cite{GuWu25} and the results obtained by Montoro-Muglia-Sciunzi lies in the difficulty of classifying one-dimensional solutions. For the local equation \eqref{eq. main-i}, one-dimensional solutions naturally satisfy an ODE, which can be analyzed by multiplying $u'(t)$ on both sides and integrating. This method is obviously invalid for nonlocal equations.

The purpose of this paper is to establish classification results for the problem \eqref{eq. main}, namely the fractional Lane-Emden-Fowler problem with singular nonlinearity in the half space. In this paper, we also improve the classification result in \cite{GuWu25} by classifying all one-dimensional solutions, even if global solutions cannot be explicitly written by closed form or by clear integrations.

\subsection{Main contributions}
Now we state our main result. First, we have the following existence and non-existence result.
\begin{theorem}\label{thm. nonexist}
    Let us fix some $n\in\mathbb{Z}_{+}$, $s\in(0,1)$, $\gamma\in(0,+\infty)$, and $\alpha\in\mathbb{R}$. Then, 
    there exists at least one global solution to \eqref{eq. main} 
    if and only if \begin{equation}\label{eq. existence requirement}
        -2s<\alpha<(\gamma-1)s.
    \end{equation}
\end{theorem}
Next, we construct global solutions to \eqref{eq. main}.
\begin{theorem}\label{thm. construction}
    Assume that \eqref{eq. existence requirement} holds. Then:
    \begin{itemize}
        \item[(1)] There exists a $\frac{\alpha+2s}{1+\gamma}$-homogeneous global solution to \eqref{eq. main} in the form:
        \begin{equation*}
            U_{0}(x)=U_{0}(x_{n})=C(s,\gamma,\alpha)x_{n}^{\beta},\quad\mbox{where }\beta=\frac{\alpha+2s}{1+\gamma}.
        \end{equation*}
        \item[(2)] For all $K>0$, there exists a $U_{K}(x)=U_{K}(x_{n})$ solving \eqref{eq. main} with
        \begin{equation*}
            \lim_{t\to\infty}\frac{U_{K}(t)}{t^{s}}=K.
        \end{equation*}
    \end{itemize}
\end{theorem}
We have the following further description of $U_{K}$'s constructed in Theorem~\ref{thm. construction}.
\begin{theorem}\label{thm. further property}
    Assume that \eqref{eq. existence requirement} holds, and let $U_{K}$'s for $K\geq0$ be the one dimensional solutions constructed in Theorem~\ref{thm. construction}. Then:
    \begin{itemize}
        \item[(1)] For $t\geq0$ and $K>0$, we have
        \begin{equation*}
            \max\{K t^{s},U_{0}(t)\}\leq U_{K}(t)\leq K t^{s}+U_{0}(t).
        \end{equation*}
        \item[(2)] The solutions are mutually scaling-related. Precisely speaking, for every $\lambda>0$,
        \begin{equation*}
            U_{\lambda K}(t)=\lambda^{-\frac{\alpha+2s}{s(\gamma-1)-\alpha}}U_{K}(\lambda^{\frac{1+\gamma}{s(\gamma-1)-\alpha}}t).
        \end{equation*}
        \item[(3)] The global solutions $U_{K}(x_{n})$'s are strictly increasing in $x_{n}$ for $x_{n}\geq0$.
    \end{itemize}
\end{theorem}
Finally, we give the full classification of global solutions.
\begin{theorem}\label{thm. classification}
    Assume that \eqref{eq. existence requirement} holds. Then:
    \begin{itemize}
        \item[(1)] All solutions $u(x)$ to \eqref{eq. main} must be one-dimensional, i.e.:
        \begin{equation*}
            u(x)=u(x_{n}).
        \end{equation*}
        \item[(2)] There are no global solutions to \eqref{eq. main} other than $U_{0}(x_{n})$ and $U_{K}(x_{n})$'s in Theorem~\ref{thm. construction}.
    \end{itemize}
\end{theorem}

In fact, its non-fractional counterpart also holds, and we state it below.
\begin{theorem}
    All statements in Theorem~\ref{thm. nonexist}, Theorem~\ref{thm. construction}, Theorem~\ref{thm. further property}, and Theorem~\ref{thm. classification} also hold in the case $s=1$.
\end{theorem}
\begin{proof}
    The methods in proving these theorems are also valid for the local equation, so we omit the proof.
\end{proof}

\subsection{Key idea}
Our method in proving Theorem~\ref{thm. classification} involves the following key ingredients:
\begin{itemize}
    \item Asymptotic growth rate estimate,
    \item One-dimensional symmetry,
    \item Construction of global solutions and classification of one-dimensional solutions.
\end{itemize}

First, we obtain the asymptotic growth rate estimate of a global solution $u(x)$ using the integral equation of \eqref{eq. main}. Recall that the Green function for $(-\Delta)^{s}$ in the half space $\mathbb{R}_{+}^{n}$ can be expressed by:
\begin{equation*}
    G^{n}_{+}(x,z)=\kappa(n,s)|x-z|^{2s-n}\int_0^{\frac{4x_nz_n}{|x-z|^2}}\frac{b^{s-1}}{(b+1)^{\frac{n}{2}}}db,\,\ \kappa(n,s)=\frac{\Gamma(\frac{n}{2})}{2^{2s}\pi^{\frac{n}{2}}\Gamma^2(s)}.
\end{equation*}
With the Green function, it can be verified that the problem \eqref{eq. main} in the half space is equivalent to the following integral equation:
\begin{equation}\label{grn}
u(x)=K(x_n)_+^s+\int_{\R^n_+}G^n_+(x,z)\frac{x_{n}^{\alpha}}{u(z)^{\gamma}}dz,
\end{equation}
and we are able to establish the following key estimate for the solutions.
\begin{lemma}\label{lem. u is close to a linear function}
Let $\gamma>0$ and assume that \eqref{eq. existence requirement} holds. Let $u$ be a solution of \eqref{eq. main}, then:
\begin{itemize}
    \item[(1)] There exists some $c=c(n,s,\gamma,\alpha)$ such that
    \begin{equation*}
        u(x)\geq c x_{n}^{\frac{\alpha+2s}{1+\gamma}}\quad\mbox{in }\mathbb{R}^{n}_{+}.
    \end{equation*}
    \item[(2)] There exists some uniform $C=C(n,s,\gamma,\alpha)$, and some non-uniform constant $K=K(u)$, which depends on each global solution $u$, such that
    \begin{equation}\label{accurate-estimate}
|u(x)-K x_n^s|\leq C x_n^{\frac{\alpha+2s}{\gamma+1}}\quad\mbox{in }\mathbb{R}_{+}^{n}.
\end{equation}
\end{itemize}
\end{lemma}

Second, we apply the method of moving planes to the Kelvin transform of $u(x)$, see subsection~\ref{subsection: MMP}. This proves Theorem~\ref{thm. classification} (1), i.e., the one-dimensional symmetry of $u(x)$.
Finally, we construct global solutions through a limiting argument, which is fundamentally different from the ODE analysis in \cite{MMS,MMS1}. We let $U_{K,b}(t)$ be the solution to the problem \eqref{eq. slope solution in bounded interval} in a finite interval $[0,b]$, and show that $\displaystyle U_{K}(t)=\lim_{b\to\infty}U_{K,b}(t)$ is a global solution. Here, for the nonlocal equation, it is actually not a trivial fact that $(-\Delta)^{s}U_{K,b}$ converges to $(-\Delta)^{s}U_{K}$ even in the pointwise sense. To overcome such a difficulty, we obtain a version of the monotone convergence theorem (see Lemma~\ref{lem. monotone convergence}), and use the observation that $U_{K,b}(t)$ is decreasing in the index $b$. With these one-dimensional solutions, we can apply Lemma~\ref{lem. u is close to a linear function}, and see that if $u(x)$ satisfies \eqref{accurate-estimate}, then $u(x)\equiv U_{K}(x_{n})$ with the same asymptotic $s$-order slope $K$.

\section{Preliminaries}
In this section, we prove some preliminary lemmas.
\subsection{Monotone convergence lemma for \texorpdfstring{$(-\Delta)^s$}{Lg}}
We start with a monotone convergence lemma:
\begin{lemma}[Monotone convergence]\label{lem. monotone convergence}
    Let $\sigma>0$ be small and
    \begin{equation*}
        2s+\sigma\in(0,2).
    \end{equation*}
    Assume that a sequence of $C^{1,1}_{loc}(B_{1})$ functions $u_{k}:\mathbb{R}^{n}\to\mathbb{R}$ satisfies the uniform $\mathcal{L}_{2s}$ bound:
    \begin{equation*}
        \|u_{k}\|_{\mathcal{L}_{2s}(\mathbb{R}^{n})}\leq C,
    \end{equation*}
    and is pointwise increasing (or decreasing), meaning that
    \begin{equation*}
        u_{k}(x)\leq (\mbox{or }\geq)\ u_{k+1}(x),\quad\mbox{for all }x\in\mathbb{R}^{n}.
    \end{equation*}
    If $u_{k}(x)$ converges to $u_{\infty}(x)$ in $C^{2s+\sigma}_{loc}(B_{1})$ sense, then $u_{\infty}(x)$ satisfies the equation
    \begin{equation*}
        (-\Delta)^{s}u_{\infty}(x)=\lim_{k\to\infty}(-\Delta)^{s}u_{k}(x)\quad\mbox{in }B_{1}.
    \end{equation*}
\end{lemma}
\begin{proof}
    We write $v_{k}=u_{k}-u_{\infty}$ and divide the expression of $(-\Delta)^{s}(u_{k}-u_{\infty})$ into several parts:
    \begin{align*}
        &(-\Delta)^{s}u_{k}(x)-(-\Delta)^{s}u_{\infty}(x)=C_{n,s}PV\int\frac{v_{k}(x)-v_{k}(y)}{|x-y|^{n+2s}}dy\\
        =&C_{n,s}\Big\{PV\int_{B_{r}(x)}+\int_{B_{R}\setminus B_{r}(x)}+\int_{B_{R}^{c}}\Big\}\frac{v_{k}(x)-v_{k}(y)}{|x-y|^{n+2s}}dy=:I_{1}+I_{2}+I_{3}.
    \end{align*}
    Here, we choose $r=\frac{1-|x|}{2}$, so that $B_{r}(x)$ is totally contained in $B_{1}$.
    
    For every $\delta>0$, we intend to show that $|(-\Delta)^{s}u_{k}(x)-(-\Delta)^{s}u_{\infty}(x)|\leq\delta$ for sufficiently large $k$. To do this, we first choose a sufficiently large outer radius $R$, such that
    \begin{equation*}
        \Big(|u_{1}(x)|+|u_{\infty}(x)|\Big)\int_{B_{R}^{c}}\frac{1}{|x-y|^{n+2s}}dy+\|u_{1}\|_{\mathcal{L}_{2s}(B_{R}^{c})}+\|u_{\infty}\|_{\mathcal{L}_{2s}(B_{R}^{c})}\leq\frac{\delta}{6C_{n,s}},
    \end{equation*}
    where we have used Lebesgue's Monotonic Convergence Theorem to make sure $u_{\infty}\in\mathcal{L}_{2s}(\mathbb{R}^{n})$. For such a large $R$, we are sure that $|I_{3}|\leq\delta/3$ for all $k$'s because:
    \begin{align*}
        |I_{3}|\leq&C_{n,s}\int_{B_{R}^{c}}\frac{|u_{k}(x)|+|u_{\infty}(x)|+|u_{k}(y)|+|u_{\infty}(y)|}{|x-y|^{n+2s}}dy\\
        \leq&C_{n,s}\int_{B_{R}^{c}}\frac{(|u_{1}(x)|+|u_{\infty}(x)|)+|u_{\infty}(x)|+(|u_{1}(y)|+|u_{\infty}(y)|)+|u_{\infty}(y)|}{|x-y|^{n+2s}}dy\\
        \leq&2C_{n,s}\int_{B_{R}^{c}}\frac{|u_{1}(x)|+|u_{\infty}(x)|}{|x-y|^{n+2s}}dy+2C_{n,s}\int_{B_{R}^{c}}\frac{|u_{1}(y)|+|u_{\infty}(y)|}{|x-y|^{n+2s}}dy\\
        =&2C_{n,s}\Big\{(|u_{1}(x)|+|u_{\infty}(x)|)\int_{B_{R}^{c}}\frac{1}{|x-y|^{n+2s}}dy+\|u_{1}\|_{\mathcal{L}_{2s}(B_{R}^{c})}+\|u_{\infty}\|_{\mathcal{L}_{2s}(B_{R}^{c})}\Big\}\\
        \leq&2C_{n,s}\cdot\frac{\delta}{6C_{n,s}}=\delta/3.
    \end{align*}

    We next compute $I_{1}$. We use the assumption
    \begin{equation*}
        v_{k}=u_{k}-u_{\infty}\to0\quad\mbox{in }C^{2s+\sigma}_{loc}(B_{1})\subseteq C^{2s+\sigma}(B_{r}(x)),
    \end{equation*}
    and naturally assume that there exists a sequence $h_{k}\to0$ such that:
    \begin{align*}
        \left\{\begin{aligned}
            &|v_{k}(x)-v_{k}(y)|\leq h_{k}|x-y|^{2s+\sigma}&\mbox{if }0<2s+\sigma\leq1\\
            &|v_{k}(x)-v_{k}(y)-(x-y)\cdot\nabla u(x)|\leq h_{k}|x-y|^{2s+\sigma}&\mbox{if }1<2s+\sigma<2
        \end{aligned}\right.
    \end{align*}
    for all $y\in B_{r}(x)$. Then, for sufficiently large $k$, we have
    \begin{equation*}
        |I_{1}|\leq C_{n,s}\int_{B_{r}(x)}\frac{h_{k}|x-y|^{2s+\sigma}}{|x-y|^{n+2s}}dy=C_{n,s,\sigma,r}\cdot h_{k}\leq\delta/3.
    \end{equation*}

    Finally, we compute $I_{2}$. As $k\to\infty$, we clearly have $v_{k}(x)\to0$. Besides, we apply the Lebesgue's Monotonic Convergence Theorem to $v_{k}$ and obtain that
    \begin{equation*}
        \int_{B_{R}\setminus B_{r}(x)}\frac{|v_{k}(y)|}{|x-y|^{n+2s}}dy\to0.
    \end{equation*}
    Therefore, for sufficiently large $k$, we have $|I_{2}|\leq\delta/3$. We then reach the convergence:
    \begin{equation*}
        |(-\Delta)^{s}u_{k}(x)-(-\Delta)^{s}u_{\infty}(x)|\leq|I_{1}|+|I_{2}|+|I_{3}|\leq\delta\quad\mbox{for sufficiently large }k.
    \end{equation*}
\end{proof}

\subsection{Well-posedness}
We now prove the uniqueness and existence of $(-\Delta)^{s}u=\frac{x_{n}^{\alpha}}{u^{\gamma}}$ in a bounded open domain. First, since the right-hand side of the equation is monotonic increasing in $u$, we easily obtain the following maximal principle.
\begin{lemma}[Maximal principle]\label{lem. maximal principle}
    Let $\Omega$ be a bounded open set in $\mathbb{R}^{n}_{+}$, and let $u,v\geq0$ be two $C^{1,1}_{loc}(\Omega)\cap C(\overline{\Omega})\cap\mathcal{L}_{2s}(\mathbb{R}^{n})$ functions such that $u\geq v$ in $\Omega^{c}$. If we further know
    \begin{equation*}
        (-\Delta)^{s}u\geq\frac{x_{n}^{\alpha}}{u^{\gamma}},\quad(-\Delta)^{s}v\leq\frac{x_{n}^{\alpha}}{v^{\gamma}}\quad\mbox{in }\Omega,
    \end{equation*}
    where $\gamma>0$, then $u\geq v$ in $\Omega$.
\end{lemma}
\begin{proof}
    Suppose otherwise, then there exists an interior point $x^0\in\Omega$, such that
    \begin{equation*}
        u(x^0)-v(x^0)=\min_{x\in\Omega}\Big(u(x)-v(x)\Big)=-h<0.
    \end{equation*}
    On the one hand, since $w(x):=u(x)-v(x)\geq-h$ elsewhere, we have
    \begin{equation*}
        (-\Delta)^{s}w(x^0)=C_{n,s}PV\int\frac{w(x^0)-w(y)}{|x-y|^{n+2s}}dy\leq C_{n,s}PV\int\frac{0}{|x-y|^{n+2s}}dy=0.
    \end{equation*}
    On the other hand, as $u(x^0)<v(x^0)$, we have
    \begin{equation*}
        (-\Delta)^{s}w(x^0)=(-\Delta)^{s}u(x^0)-(-\Delta)^{s}v(x^0)\geq\frac{x_{n}^{\alpha}}{u(x^0)^{\gamma}}-\frac{x_{n}^{\alpha}}{v(x^0)^{\gamma}}>0.
    \end{equation*}
    Then, we have reached a contradiction.
\end{proof}
Using Lemma~\ref{lem. maximal principle}, we obtain a pointwise lower bound estimate by constructing a lower barrier.
\begin{lemma}[Pointwise lower bound]\label{lem. lower bound of a local solution}
Let $\gamma>0$ and $u$ be a solution of $(-\Delta)^{s}u=\frac{x_{n}^{\alpha}}{u^{\gamma}}$ in $B_{r}(h\vec{e_{n}})$ for some $0<r\leq h$. Then there exists a constant $c=c(n,s,\gamma,\alpha)$ such that
\begin{equation*}
    u(h\vec{e_{n}})\geq c \cdot h^{\frac{\alpha}{1+\gamma}}r^{\frac{2s}{1+\gamma}}.
\end{equation*}
\end{lemma}
\begin{proof}
Let $\psi(x)=(1-|x|^2)_{+}^{s}$, then $(-\Delta)^{s}\psi(x)\equiv b$ in $B_{1}$ for some constant $b>0$. Denote
\begin{equation*}
    \phi(x)=A\cdot\psi(\frac{x-h\vec{e_{n}}}{r/2}),\quad\mbox{with }A=\left(\frac{2^{-2s-|\alpha|}h^{\alpha}r^{2s}}{b}\right)^\frac{1}{1+\gamma}.
\end{equation*}
Then, we have for $x\in B_{r/2}(h\vec{e_{n}})$
\begin{equation*}
(-\Delta)^s \phi(x)=(\frac{r}{2})^{-2s}A\cdot b=\frac{2^{-|\alpha|}h^{\alpha}}{A^\gamma}\leq \frac{\min_{y\in B_{r/2}(h\vec{e_{n}})}y_{n}^{\alpha}}{A^\gamma}\leq\frac{x_{n}^{\alpha}}{\phi^\gamma}.
\end{equation*}
Notice that $u(x)\geq0=\phi(x)$ outside $B_{r/2}(h\vec{e_{n}})$, so by Lemma~\ref{lem. maximal principle}, we see
\begin{equation*}
    u(h\vec{e_{n}})\geq\phi(h\vec{e_{n}})=A=\left(\frac{2^{-2s-|\alpha|}h^{\alpha}r^{2s}}{b}\right)^\frac{1}{1+\gamma}=c(n,s,\gamma,\alpha)\cdot h^{\frac{\alpha}{1+\gamma}}r^{\frac{2s}{1+\gamma}}.
\end{equation*}
\end{proof}
Now we prove the well-posedness of the equation $\displaystyle(-\Delta)^{s}u=\frac{x_n^{\alpha}}{u^{\gamma}}$ in a bounded open domain.
\begin{lemma}[Well-posedness]\label{lem. well posedness}
    Let $\Omega=B_{r}^{+}$ in $\mathbb{R}^{n}_{+}$. Let $\varphi\geq0$ be a continuous function in $\mathbb{R}^{n}_{+}$ with $\varphi\equiv0$ in $\overline{\mathbb{R}^{n}_{-}}$, and assume that
    \begin{equation*}
        \int_{\mathbb{R}^{n}}\frac{\varphi(x)}{1+|x|^{n+2s}}dt<\infty.
    \end{equation*}
    Then there exists a unique solution $u\in C^{1,1}_{loc}(\Omega)$ satisfying $(-\Delta)^{s}u(x)=\frac{x_{n}^{\alpha}}{u(x)^{\gamma}}$ in $\Omega$, and $u$ continuously extends to $\varphi$ outside.
\end{lemma}
\begin{proof}
    The uniqueness is guaranteed by Lemma~\ref{lem. maximal principle}, so we focus on the existence part. The proof is divided into two parts: we first prove the existence of a solution $u_{\varepsilon}$ if the boundary condition $\varphi$ is replaced by $\varphi+\varepsilon$, then we send $\varepsilon$ to zero and find a solution $u$.
    \begin{itemize}
        \item Step 1: Let $\varphi_{\varepsilon}=\varphi+\varepsilon$ for $\varepsilon>0$. We use the continuity method to construct $u_{\varepsilon}$, meaning that we intend to find a family solutions $v_{\varepsilon,\lambda}$ ($\lambda\in[0,1]$) to the problem:
        \begin{equation}\label{eq. continuity method problem}
            \left\{\begin{aligned}
                &(-\Delta)^{s}v_{\varepsilon,\lambda}(x)=\lambda\frac{x_{n}^{\alpha}}{v_{\varepsilon,\lambda}(x)^{\gamma}}&\mbox{in }&\Omega\\
                &v_{\varepsilon,\lambda}(x)=\varphi_{\varepsilon}&\mbox{in }&\Omega^{c}
            \end{aligned}\right..
        \end{equation}
        The existence of $v_{\varepsilon,0}$ is obvious, and if some $v_{\varepsilon,\lambda}$ exists, we always have $v_{\varepsilon,\lambda}\geq\varepsilon$ since it is fractional subharmonic in $I$. Let $\mathcal{S}\subseteq[0,1]$ be defined as follows:
        \begin{equation*}
            \mathcal{S}:=\{\Lambda\in[0,1]:\mbox{ for all }\lambda\in[0,\Lambda],\ \eqref{eq. continuity method problem}\mbox{ admits a solution}\}.
        \end{equation*}
        It is clear that either $\mathcal{S}=[0,\Lambda)$, or $\mathcal{S}=[0,\Lambda]$. Here, $\Lambda\in[0,1]$, and we treat $[0,0)=\emptyset$, $[0,0]=\{0\}$.
        
        From the existence of $v_{\varepsilon,0}$, we see $\mathcal{S}$ is non-empty as it at least includes $0$.
        
        Next, we show that $\mathcal{S}$ is open with respect to the induced topology of $[0,1]$. To see this, we need to show that if $[0,\Lambda]\subseteq\mathcal{S}$ for some $0\leq\Lambda<1$, then $[0,\Lambda+\delta]\subseteq\mathcal{S}$ for some $\delta>0$. The idea is to use the implicit function theorem. By taking the differential of \eqref{eq. continuity method problem} in $\lambda$, we have
        \begin{equation}\label{eq. continuity derivative equation}
            \left\{\begin{aligned}
                &(-\Delta)^{s}\frac{d v_{\varepsilon,\lambda}}{d\lambda}(x)=\frac{x_{n}^{\alpha}}{v_{\varepsilon,\lambda}(x)^{\gamma}}&\mbox{in }&\Omega\\
                &\frac{d v_{\varepsilon,\lambda}}{d\lambda}(x)=0&\mbox{in }&\Omega^{c}
            \end{aligned}\right..
        \end{equation}
        As we already know $v_{\varepsilon,\lambda}\geq\varepsilon$ everywhere, we see the right-hand side $\frac{x_{n}^{\alpha}}{v_{\varepsilon,\lambda}(x)^{\gamma}}$ of \eqref{eq. continuity derivative equation} is bounded from above by $C x_{n}^{\alpha}$. Since we assume $\alpha>-2s$, we see the derivative $\frac{d v_{\varepsilon,\lambda}}{d\lambda}(x)$ must exist and is bound and positive. Therefore, we can make sure that \eqref{eq. continuity method problem} is solvable when $\lambda$ is in the vicinity of $\Lambda$.

        Finally, we show that $\mathcal{S}$ is closed. In particular, if $[0,\Lambda)\subseteq\mathcal{S}$ for some $\Lambda\in[0,1]$, we need to show that \eqref{eq. continuity method problem} is solvable for $\lambda=\Lambda$. We let $\Lambda_{k}=\Lambda-k^{-1}$, then $\Lambda_{k}\in[0,\Lambda)\subseteq\mathcal{S}$ and $\Lambda_{k}\to\Lambda$. Since \eqref{eq. continuity derivative equation} implies $\frac{d v_{\varepsilon,\lambda}}{d\lambda}(x)\geq0$, we have that $v_{\varepsilon,\Lambda_{k}}(x)$ is pointwise increasing with respect to $k$. Therefore, a pointwise convergence of $v_{\varepsilon,\Lambda_{k}}(x)$ to some function $u(x)$ is guaranteed. It then suffices to show that the pointwise limit $u(x)$ solves \eqref{eq. continuity method problem} for $\lambda=\Lambda$, i.e.. $u(x)=v_{\varepsilon,\Lambda}(x)$.

        For each $k$, we use the fact that $v_{\varepsilon,\Lambda_{k}}\geq\varepsilon$ to bound the right-hand side $\lambda\frac{x_{n}^{\alpha}}{v_{\varepsilon,\lambda}(x)^{\gamma}}$ of \eqref{eq. continuity method problem} uniformly. From the assumption $\alpha>-2s$, we have the uniform estimate:
        \begin{equation*}
            \|v_{\varepsilon,\Lambda_{k}}-v_{\varepsilon,0}\|_{C^{2\sigma}_{0}(\overline{\Omega})}\leq C_{s,\alpha,b,\varepsilon}.
        \end{equation*}
        In particular, passing to the limit, we see that $u(x)$ extends continuously to $\varphi_{\varepsilon}$ outside $\Omega$. Here, $\sigma>0$ is some sufficiently small exponent and we assume that
        \begin{equation*}
            2s+2\sigma\in(0,1)\cup(1,2)
        \end{equation*}
        for simplicity. Such an estimate also improves the regularity of $\lambda\frac{x_{n}^{\alpha}}{v_{\varepsilon,\lambda}(x)^{\gamma}}$, and we have the following uniform $C^{2s+2\sigma}_{loc}(\Omega)$ estimate:
        \begin{equation}\label{eq. approaching sequence local uniform bound}
            \|v_{\varepsilon,\Lambda_{k}}\|_{C^{2s+2\sigma}(\Omega')}\leq C_{s,\alpha,b,\varepsilon}(\Omega'),\quad\mbox{for all }\Omega'\subset\subset\Omega.
        \end{equation}
        With such an estimate, we show that $u(x)$ solves the PDE in \eqref{eq. continuity method problem} for $\lambda=\Lambda$. We apply Lemma~\ref{lem. monotone convergence} near each interior point of $\Omega$, in which we have chosen a subsequence of $v_{\varepsilon,\Lambda_{k}}$, so that \eqref{eq. approaching sequence local uniform bound} implies that the subsequence converges to $u(x)$ in $C^{2s+\sigma}_{loc}$. This means, up to a subsequence,
        \begin{equation}\label{eq. limit equation computation example}
            (-\Delta)^{s}u(x)=\lim_{k\to\infty}(-\Delta)^{s}v_{\varepsilon,\Lambda_{k}}(x)=\lim_{k\to\infty}\Lambda_{k}\frac{x_{n}^{\alpha}}{v_{\varepsilon,\lambda}(x)^{\gamma}}=\Lambda\frac{x_{n}^{\alpha}}{u(x)^{\gamma}}\mbox{ for }x\in\Omega.
        \end{equation}
        Therefore, $u(x)$ solves \eqref{eq. continuity method problem} and $u=v_{\varepsilon,\Lambda_{k}}$.

        Since $0\in\mathcal{S}$, and $\mathcal{S}$ is both open and closed with respect to the topology in $[0,1]$, we see $v_{\varepsilon,1}=u_{1}$ exists.
        \item Step 2: Now we send $\varepsilon\to0$. By Lemma~\ref{lem. maximal principle}, $u_{\varepsilon}$ is pointwise decreasing as $\varepsilon$ decreases, so $u_{\varepsilon}$ has a pointwise limit $u_{0}$. We need to show that $u_{0}$ is a solution to $(-\Delta)^{s}u=\frac{x_{n}^{\alpha}}{u(x)^{\gamma}}$ in $\Omega$, and $u$ continuously extends to $\varphi$. First, using Lemma~\ref{lem. lower bound of a local solution}, we have a pointwise upper bound of $\frac{x_{n}^{\alpha}}{u(x)^{\gamma}}$, which yields a uniform interior local estimate for $u_{\varepsilon}$:
        \begin{equation*}
            \|u_{\varepsilon}\|_{C^{2s+2\sigma}(\Omega')}\leq C_{s,\alpha,b,\varepsilon}(\Omega'),\quad\mbox{for all }\Omega'\subset\subset\Omega.
        \end{equation*}
        Then, up to a subsequence, we have $u_{\varepsilon}\to u_{0}$ in $C^{2s+\sigma}_{0}(\Omega)$ sense. Using Lemma~\ref{lem. monotone convergence},
        \begin{equation*}
            (-\Delta)^{s}u_{0}(x)=\lim_{k\to\infty}(-\Delta)^{s}u_{\varepsilon}(x)=\lim_{k\to\infty}\frac{x_{n}^{\alpha}}{u_{\varepsilon}(x)^{\gamma}}=\frac{x_{n}^{\alpha}}{u_{0}(x)^{\gamma}}\mbox{ for }x\in\Omega.
        \end{equation*}

        Next, since $u_{\varepsilon}$ are all fractional sub-harmonic, we can bound all $u_{\varepsilon}$'s from below using the harmonic replacement of $\varphi$:
        \begin{equation*}
            \left\{\begin{aligned}
                &(-\Delta)^{s}h=0&\mbox{in }&\Omega\\
                &h=\varphi&\mbox{in }&\Omega^c\\
            \end{aligned}\right..
        \end{equation*}
        It then follows that
        \begin{equation*}
            \liminf_{x\to y\in\partial\Omega}u_{0}(x)\geq\liminf_{\varepsilon\to0}\liminf_{x\to y\in\partial\Omega}u_{\varepsilon}(x)\geq\liminf_{x\to y\in\partial\Omega}h(x)\geq\varphi(y).
        \end{equation*}
        Conversely, we show
        \begin{equation*}
            \limsup_{x\to y\in\partial\Omega}u_{0}(x)\leq\varphi(y).
        \end{equation*}
        For every fixed $y\in\partial\Omega$ and $\varepsilon>0$, by the continuity of $u_{\varepsilon/2}(x)$, there exists some $\delta>0$ such that
        \begin{equation}
            \Big|u_{\varepsilon/2}(x)-\big(\varphi(y)+\frac{\varepsilon}{2}\big)\Big|\leq\frac{\varepsilon}{2},\quad\mbox{if }|x-y|\leq\delta.
        \end{equation}
        Then, for $|x-y|\leq\delta$, we have
        \begin{equation}
            u_{0}(x)-\varphi(y)\leq u_{\varepsilon/2}(x)-\varphi(y)\leq\varepsilon.
        \end{equation}
        Then we see $u=u_{0}$ solves $(-\Delta)^{s}u(x)=\frac{x_{n}^{\alpha}}{u(x)^{\gamma}}$ in $\Omega$, and $u$ continuously extends to $\varphi$ outside.
    \end{itemize}
\end{proof}

\section{Existence and non-existence}
\subsection{Existence}
\begin{lemma}\label{lem. homogeneous solution derivative expression}  (\cite[Lemma 9.2.2]{CLM2020})
For $t\in\mathbb{R}$, let $t^{+}:=\max\{t,0\}$. Then for $0<\beta<2s$, we have
\begin{equation*}
(-\Delta)_1^s(t_+^\beta)=K_\beta t_+^{\beta-2s}\ \mb{in}\ t>0,
\end{equation*}
where
\begin{eqnarray}
K_\beta:=C_{1,s}\int_0^1\frac{(
{\tau}^{s}-{\tau}^\beta)(1-{\tau}^{s-\beta-1})}{(1-\tau)^{1+2s}}d\tau\ \left\{\begin{array}{ll}>0, &\mb{if }\ 0<\beta<s,\\ =0, &\mb{if }\ \beta=s,\\ <0, &\text{if }\ s<\beta<2s.\end{array}\right.
\end{eqnarray}
On the other hand, when $\beta\geq2s$, then $t_+^\beta$ is not in $\mathcal{L}_{2s}(\mathbb{R})$, and $(-\Delta)_1^s(t_+^\beta)$ is not well-defined.
\end{lemma}
With Lemma~\ref{lem. homogeneous solution derivative expression}, we are able to find a homogeneous solution to \eqref{eq. main} when \eqref{eq. existence requirement} holds.
\begin{proof}[Proof of Theorem~\ref{thm. construction} (1)]
    Let $K_\beta$ be the constant given in Lemma~\ref{lem. homogeneous solution derivative expression}. If \eqref{eq. existence requirement} holds, then
    \begin{equation}\label{particular soln}
U_{0}(t)={C}_{\gamma,s,\alpha}t_+^{\frac{\alpha+2s}{\gamma+1}},\quad C_{\gamma,s,\alpha}=(K_{\frac{\alpha+2s}{\gamma+1}})^{-\frac{1}{\gamma+1}},
\end{equation}
is a homogeneous one-dimensional solution to \eqref{eq. main}.
\end{proof}
The construction of other solutions in Theorem~\ref{thm. construction} (2) is more involved.
\begin{proof}[Proof of Theorem~\ref{thm. construction} (2)]
    We define a sequence of functions $U_{K,b}(t)$ satisfying
    \begin{equation}\label{eq. slope solution in bounded interval}
        \left\{\begin{aligned}
            &(-\Delta)^{s}U_{K,b}(t)=\frac{t^{\alpha}}{U_{K,b}(t)^{\gamma}}&\mbox{in }(0,b)\\
            &U_{K,b}(t)=U_{0}(t)+K\cdot(t_{+})^{s}&\mbox{elsewhere}
        \end{aligned}\right.
    \end{equation}
    Since $K\cdot(t_{+})^{s}$ is fractional harmonic, it is a sub-solution to \eqref{eq. slope solution in bounded interval}, and since
    \begin{equation*}
        (-\Delta)^{s}U_{0}(t)+K\cdot(t_{+})^{s}=(-\Delta)^{s}U_{0}(t)=\frac{t^{\alpha}}{U_{0}(t)^{\gamma}}\geq\frac{t^{\alpha}}{(U_{0}(t)+K\cdot(t_{+})^{s})^{\gamma}},
    \end{equation*}
    we see $(U_{0}(t)+Kt^{s})_{+}$ is a super-solution to \eqref{eq. slope solution in bounded interval}. In conclusion, we have for every $t\geq0$ that
    \begin{equation}\label{eq. bounded slope solution also bounded}
        Kt^{s}\leq U_{K,b}(t)\leq Kt^{s}+U_{0}(t).
    \end{equation}
    
    We next notice that $U_{K,b}(t)$ is pointwise decreasing in $b$. In fact, we arbitrarily choose $b_{1}<b_{2}$, then as $U_{K,b_{2}}(t)\leq K\cdot(t_{+})^{s}+U_{0}(t)$, we see that $U_{K,b_{1}}(t)$ and $U_{K,b_{2}}(t)$ satisfies the same equation in $(0,b_{1})$, but the exterior condition of $U_{K,b_{1}}(t)$ is larger than that of $U_{K,b_{2}}(t)$. By Lemma~\ref{lem. maximal principle}, we have $U_{K,b_{1}}(t)\geq U_{K,b_{2}}(t)$.

    From the bound \eqref{eq. bounded slope solution also bounded} and the monotonicity of $U_{K,b}(t)$ in $b$, we see
    \begin{equation*}
        U_{K,b}(t)\to U_{K,\infty}(t)
    \end{equation*}
    for each $t$, with $U_{K,\infty}(t)$ also satisfying \eqref{eq. bounded slope solution also bounded} and extending trivially to $(-\infty,0]$. Using Lemma~\ref{lem. monotone convergence} and arguing similar to \eqref{eq. limit equation computation example}, we get that $U_{K,\infty}(t)$ is a one-dimensional global solution to \eqref{eq. main}. Then, naturally, we choose $U_{K}(t):=U_{K,\infty}(t)$ and the construction is completed. By sending $b\to\infty$ in \eqref{eq. bounded slope solution also bounded}, we have
    \begin{equation}\label{eq. U_K almost best estimate}
        Kt^{s}\leq U_{K}(t)\leq Kt^{s}+U_{0}(t).
    \end{equation}
    Finally, since $U_{0}(t)=O(t^{\frac{\alpha+2s}{1+\gamma}})=o(t^s)$, we have $\lim_{t\to+\infty}\frac{U_{K}(t)}{t^s}=K$.
\end{proof}
\subsection{Non-existence}
In this part, we show that once the condition \eqref{eq. existence requirement} is violated, i.e.:
\begin{equation*}
    \mbox{either }\alpha\leq-2s,\mbox{ or }\alpha\geq(\gamma-1)s,
\end{equation*}
then \eqref{eq. main} does not have a global solution. Moreover, we can construct a homogeneous solution to \eqref{eq. main} once \eqref{eq. existence requirement} holds. These prove Theorem~\ref{thm. nonexist}.
\begin{proof}[Proof of Theorem~\ref{thm. nonexist}]
    When \eqref{eq. existence requirement} holds, then Theorem~\ref{thm. construction} ensures the existence of a global solution. Now we show the necessity of \eqref{eq. existence requirement}. If \eqref{eq. existence requirement} fails, then either $\alpha\leq-2s$, or $\alpha\geq(\gamma-1)s$.
    
    The simpler part is to prove the non-existence of global solutions in the case $\alpha\leq-2s$. To see this, it suffices to show the non-existence of a bounded function $u$ satisfying
    \begin{equation*}
        (-\Delta)^{s}u(x)=\frac{x_{n}^{\alpha}}{u_{r}(x)^{\gamma}}\quad\mbox{in }B_{1}^{+}.
    \end{equation*}
    Suppose otherwise, meaning that we have $u\leq C_{0}$ in $B_{1}^{+}$. Then, we see the right-hand side $\frac{x_{n}^{\alpha}}{u_{1}(x)^{\gamma}}$ is bounded from below by $c x_{n}^{\alpha}$. When $\alpha\leq-2s$, then this violates the continuity of $u$ near $\partial\mathbb{R}^{n}_{+}$.

    The harder part is the case $\alpha\geq(\gamma-1)s$. We let $u_{r}$ be the solution to
    \begin{equation}\label{eq. Br local solution}
        \left\{\begin{aligned}
            &(-\Delta)^{s}u_{r}(x)=\frac{x_{n}^{\alpha}}{u_{r}(x)^{\gamma}}&\mbox{in }B_{r}^{+}\\
            &u(x)=0&\mbox{elsewhere}
        \end{aligned}\right..
    \end{equation}
    It follows from Lemma~\ref{lem. maximal principle} that $u_{r}(x)$ is pointwise increasing in $r$, and they are smaller than any global solution to \eqref{eq. main}, if at least one global solution exists. Therefore, suppose on the contrary that \eqref{eq. main} has a global solution, then by Lemma~\ref{lem. monotone convergence} and using similar computation as in \eqref{eq. limit equation computation example}, we see that $u_{r}(x)$ converges to some $u_{\infty}(x)$, which is a global solution to \eqref{eq. main}.

    We notice that $u_{r}(x)=r^{\frac{\alpha+2s}{\gamma+1}}u_{1}(\frac{x}{r})$, then we see in particular that for every $x\in\partial B_{1}$,
    \begin{equation*}
        u_{\lambda r}(\lambda x)=\lambda^{\frac{\alpha+2s}{\gamma+1}}u_{r}(x).
    \end{equation*}
    By sending $r\to\infty$, we obtain that $u_{\infty}(\lambda x)=\lambda^{\frac{\alpha+2s}{\gamma+1}}u_{\infty}(x)$, i.e., $u_{\infty}(x)$ is homogeneous with degree $\frac{\alpha+2s}{\gamma+1}$.

    Besides, pick an arbitrary $p\in\partial\mathbb{R}^{n}_{+}$, by the domain inclusion
    \begin{equation*}
        B_{r}(x)\leq B_{2r}(x-p)\leq B_{4r}(x),\quad\mbox{for every }r>|p|
    \end{equation*}
    and Lemma~\ref{lem. maximal principle}, we have:
    \begin{equation*}
        u_{r}(x)\leq u_{2r}(x-p)\leq u_{4r}(x),\quad\mbox{for every }r>|p|.
    \end{equation*}
    We then send $r\to\infty$ and obtain $u_{\infty}(x)=u_{\infty}(x-p)$, in other words, $u_{\infty}(x)$ depends only on $x_{n}$.
    
    In conclusion, $u_{\infty}(x)$ is a global solution to \eqref{eq. main} and takes the form $u_{\infty}(x)=K x_{n}^{\frac{\alpha+2s}{\gamma+1}}$. If $\alpha\geq(\gamma-1)s$, we see either $u_{\infty}\notin\mathcal{L}_{2s}(\mathbb{R}^{n})$ (when $\frac{\alpha+2s}{\gamma+1}\geq2s$), or $u_{\infty}$ is not a solution to \eqref{eq. main} for any $K>0$ by Lemma~\ref{lem. homogeneous solution derivative expression} (notice that $K_{\beta}\leq0$ in Lemma~\ref{lem. homogeneous solution derivative expression} for $\beta=\frac{\alpha+2s}{\gamma+1}\geq s$). This results in a contradiction.
\end{proof}

\section{Asymptotic estimates}

In this section, we study the asymptotic behavior of a global solution, and prove Lemma~\ref{lem. u is close to a linear function}.
\begin{proof}[Proof of Lemma~\ref{lem. u is close to a linear function} (1)]
    It suffices to choose $r=h$ in Lemma~\ref{lem. lower bound of a local solution}.
\end{proof}
\begin{proof}[Proof of Lemma~\ref{lem. u is close to a linear function} (2)]
We need to recall the Green function for $(-\Delta)^{s}$ in $\mathbb{R}_{+}^{n}$:
\begin{eqnarray*}
G^{n}_{+}(x,z):=\left\{\begin{array}{ll}\vspace{0.5em}
\kappa(n,s)|x-z|^{2s-n}\int_0^{\frac{4x_nz_n}{|x-z|^2}}\fr{b^{s-1}}{(b+1)^{\frac{n}{2}}}db,\,\ &\mb{if}\ n\neq 2s,\\\vspace{0.8em}
\kappa(1, 1/2)\log\frac{\sqrt{x}+\sqrt{z}}{|\sqrt{x}-\sqrt{z}|}, \  &\mb{if}\ n=2s=1.\end{array}\right.
\end{eqnarray*}
Here, $\kappa(n,s)=\frac{\Gamma(\frac{n}{2})}{2^{2s}\pi^{\frac{n}{2}}\Gamma^2(s)}$. Moreover, for any $x:=(x',x_n)\in\R^n_+, z_n>0, x_n\neq z_n, $ it holds that:
\begin{equation}\label{grne}\int_{\R^{n-1}}G^n_+(x,z)dz'=c_nG^1_+(x_n,z_n),\, c_n=\pi^{\frac{1-n}{2}}\Gamma\left(\frac{n-1}{2}\right),\ n>1.\end{equation}
\begin{itemize}
    \item Step 1. In this step, we show that in the case $n>1$,
\begin{equation}\label{grn1}\int_{\R^n_+}G^n_+(x,z)\frac{z_{n}^{\alpha}}{u(z)^{\gamma}}dz\leq C(x_n)_+^{\frac{\alpha+2s}{1+\gamma}},\,x\in\R^n_+,\ \gamma>1.\end{equation}

It can be derived  from the lower bound estimate (Lemma~\ref{lem. lower bound of a local solution}) and \eqref{grne} that
\begin{equation*}
\int_{\R^{n}_+}G^n_+(x,z)\frac{z_{n}^{\alpha}}{u(z)^{\gamma}}dz\leq C\int_{\R^{n}_+}G^n_+(x,z)z_n^{\frac{\alpha-2s\gamma}{\gamma+1}}dz=c_n\int_0^\infty G^1_+(x_n,z_n)z_n^{\frac{\alpha-2s\gamma}{\gamma+1}}dz_n.
\end{equation*}
Consequently, in order to prove \eqref{grn}, it is sufficient to show that
\begin{equation}\label{grn3}\int_{0}^\infty G^1_+(t,\tau){\tau}^{\frac{\alpha-2s\gamma}{\gamma+1}}d\tau\leq Ct^{\frac{\alpha+2s}{1+\gamma}},\,t>0,\ \gamma>1.\end{equation}

In view of the expression of $G^{1}_{+}(t,\tau)$, we have
\begin{equation}\label{key est}
G^{1}_{+}(t,\tau)\leq\frac{\kappa(1,s)}{s|t-\tau|^{1-2s}}\left(\frac{4t\tau}{|t-\tau|^2}\right)^{s}\leq C\frac{t^{s}{\tau}^{s}}{|t-\tau|},\ \tau>0,\ t>0.
\end{equation}
Now by a variable transformation $\tau=t\tilde{\tau}$, we obtain
\begin{eqnarray}\label{est I2}
\begin{aligned}
\int_{0}^{\infty}G^{1}_{+}(t,\tau){\tau}^{\frac{\alpha-2s\gamma}{\gamma+1}}d\tau=&\kappa(1,s)\int_{0}^\infty |t-\tau|^{2s-1}\int_0^{\frac{4t\tau}{|t-\tau|^2}}\fr{b^{s-1}}{(b+1)^{\frac{1}{2}}}db\cdot
{\tau}^{\frac{\alpha-2s\gamma}{\gamma+1}}d\tau\\
=&\kappa(1,s)t^{2s+\frac{\alpha-2s\gamma}{\gamma+1}}\int_{0}^\infty |1-\tilde{\tau}|^{2s-1}\int_0^{\frac{4\tilde{\tau}}{|1-\tilde{\tau}|^2}}\fr{b^{s-1}}{(b+1)^{\frac{1}{2}}}db\cdot
{\tilde{\tau}}^{\frac{\alpha-2s\gamma}{\gamma+1}}d\tilde{\tau}\\
=&t^{\frac{\alpha+2s}{\gamma+1}}
\left\{\int_{0}^{\frac{1}{2}}+\int_{\frac{1}{2}}^{2}+\int_{2}^{\infty}\right\} G^1_+(1,\tilde{\tau})\tilde{\tau}^{\frac{\alpha-2s\gamma}{\gamma+1}}d\tilde{\tau}=:t^{\frac{\alpha+2s}{\gamma+1}}(I_{1}+I_{2}+I_{3}).\end{aligned}
\end{eqnarray}

Using \eqref{key est}, we can estimate the term $I_1$ in \eqref{est I2} as follows:
\begin{equation}\label{est I21}
I_{1}=\int_{0}^{\frac{1}{2}}G^1_+(1,\tilde{\tau})\tilde{\tau}^{\frac{\alpha-2s\gamma}{\gamma+1}}d\tilde{\tau}\leq C\int_{0}^{\frac{1}{2}}\frac{\tilde{\tau}^s}{|1-\tilde{\tau}|}{\tilde{\tau}}^{\frac{\alpha-2s\gamma}{\gamma+1}}d\tilde{\tau}\leq C\int_{0}^{\frac{1}{2}}\tilde{\tau}^{s+\frac{\alpha-2s\gamma}{\gamma+1}}d\tilde{\tau}\leq C,
\end{equation}
here we  have used the fact that $s+\frac{\alpha-2s\gamma}{\gamma+1}+1>0$.

For $\frac{1}{2}<\tau<2$, since
\begin{equation*}
\int_0^{\frac{4\tilde{\tau}}{|1-\tilde{\tau}|^2}}\fr{b^{s-1}}{(b+1)^{\frac{1}{2}}}db\leq\int_0^1b^{s-1}db
+\int_1^{\frac{4\tilde{\tau}}{(1-\tilde{\tau})^2}}b^{s-1-\frac{1}{2}}db\leq C\left(1+|1-\tilde{\tau}|^{1-2s}\right),
\end{equation*}
the term $I_2$ in \eqref{est I2} can be estimated as follows:
\begin{eqnarray}\label{est I22}
\begin{aligned}
I_{2}
=&
\int_{\frac{1}{2}}^{2}G^1_+(1,\tilde{\tau})\tilde{\tau}^{\frac{\alpha-2s\gamma}{\gamma+1}}d\tilde{\tau}\leq C\int_{\frac{1}{2}}^{{2}} |1-\tilde{\tau}|^{2s-1}\left(1+|1-\tilde{\tau}|^{1-2s}\right)\tilde{\tau}^{\frac{\alpha-2s\gamma}{\gamma+1}}d\tilde{\tau}\\
\leq& C\int_{\frac{1}{2}}^{{2}}\left(1+|1-\tilde{\tau}|^{2s-1}\right)\tilde{\tau}^{\frac{\alpha-2s\gamma}{\gamma+1}}d\tau\leq C.
\end{aligned}
\end{eqnarray}

For the term $I_3$ in \eqref{est I2}, the assumption on $\gamma$ implies $s-1+\frac{\alpha-2s\gamma}{\gamma+1}<-1$, which, together with the estimate \eqref{key est}, gives rise to
\begin{equation}\label{est I23}
I_{3}=\int_{2}^{{\infty}}G^1_+(1,\tilde{\tau})\tilde{\tau}^{\frac{\alpha-2s\gamma}{\gamma+1}}d\tilde{\tau}\leq C\int_{2}^{{\infty}}\frac{\tilde{\tau}^s}{\tilde{\tau}-1}\tilde{\tau}^{\frac{\alpha-2s\gamma}{\gamma+1}}d\tilde{\tau}\leq C\int_{2}^{{\infty}}\tilde{\tau}^{s-1+\frac{\alpha-2s\gamma}{\gamma+1}}d\tilde{\tau}\leq C.
\end{equation}
Substituting  \eqref{est I21}, \eqref{est I22} and \eqref{est I23}  into \eqref{est I2}, we deduce that
\[\int_{0}^\infty G^1_+(t,\tau){\tau}^{\frac{\alpha-2s\gamma}{\gamma+1}}d\tau\leq Ct^{\frac{\alpha+2s}{1+\gamma}},\ \forall t>0,\]
Hence, we have verified \eqref{grn3} in the case $n>1$.

\item Step 2. This step is devoted to the proof of
\begin{equation}\label{ubd1}
|u(x)-K x_{n}^{s}|\leq C x_{n}^{\frac{\alpha+2s}{\gamma+1}},\quad x\in\mathbb{R}_{+}^{n},\quad n>1.
\end{equation}
We set
\begin{equation*}
    v(x):=\int_{\R^n_+}G^n_+(x,z)\frac{z_{n}^{\alpha}}{u(z)^{\gamma}}dz.
\end{equation*}
It then follows that $u(x)-v(x)$ is fractional-harmonic in $\mathbb{R}^{n}_{+}$ and vanishes in $\mathbb{R}^{n}_{-}$. Moreover, by \eqref{grn1}, we know that $u(x)-v(x)\geq-C x_n^{\frac{\alpha+2s}{\gamma+1}}$ in $\mathbb{R}^{n}_{+}$. We claim that
\begin{equation}\label{eq. claim in integral and differential equivalence}
    u(x)-v(x)\geq-x_n^{s},\quad\mbox{for all }x\in\mathbb{R}^{n}_{+}.
\end{equation}
If \eqref{eq. claim in integral and differential equivalence} is true, then by the Liouville theorem of half-plane $s$-harmonic functions, we have $u(x)-v(x)\equiv K x_{n}^{s}$ for some constant $K\geq0$. Then \eqref{ubd1} directly follows.

To prove \eqref{eq. claim in integral and differential equivalence}, we set
\begin{equation*}
    w(x)=u(x)-v(x)+(x_n)_{+}^{s},
\end{equation*}
and our goal is to show $w\geq0$ everywhere.

In fact, it follows that $w(x)$ is fractional-harmonic in $\mathbb{R}^{n}_{+}$ and $w(x)\geq-M$ for some $M$. Moreover, $\{x:w(x)\leq0\}$ is contained in a stripe $\{x:\ 0<x_{n}<H\}$ for some $H$.

We consider a class of barrier functions for $x\in\{x:\ -2H<x_{n}<2H\}$:
\begin{equation*}
    \mathcal{B}(x,y',h):=-h\cdot(1-\frac{x_{n}^{2}}{16H^{2}})_{+}^{s}\cdot\Big(1-\frac{\varepsilon}{1+\varepsilon|x'-y'|^{2}}\Big).
\end{equation*}
Here, $h>0$, $y'\in\mathbb{R}^{n-1}$, and $\varepsilon$ is chosen so small such that
\begin{equation*}
    (-\Delta)^{s}_{x}\mathcal{B}(x,y',h)<0\quad\mbox{for all }x\in\{x:\ -2H<x_{n}<2H\}.
\end{equation*}

Suppose that $\displaystyle\inf_{\{x:\ 0<x_{n}<H\}}w=-M<0$, then there exists some $y\in\{x:\ 0<x_{n}<H\}$, such that $u(y)\leq-(1-\frac{\varepsilon}{2})M$. We translate the graph of $\mathcal{B}(x,y',h)$ to $\mathcal{B}(x-y_{n},y',h)$ and start from $h=100M$. Since $u(y)\leq-(1-\frac{\varepsilon}{2})M$, we know that there exists some $h_{0}$, so that as $h$ decreases, the graph of $\mathcal{B}(x-y_{n},y',h)$ contacts the graph of $w(x)$ for the first time at $h=h_{0}$, and at least one contacting point (say, $z^{*}$) is in the interior of the stripe $\{x:\ 0<x_{n}<H\}$. As
\begin{equation*}
    (-\Delta)^{s}_{x}\mathcal{B}(z^{*}-y_{n},y',h_{0})<0,
\end{equation*}
we have $(-\Delta)^{s}w(z^{*})<0$, which is a contradiction. This proves \eqref{eq. claim in integral and differential equivalence}.
\item Step 3. We have already shown Lemma~\ref{lem. u is close to a linear function} when $n>1$, so it remains to consider the case $n=1$. This time, for the solution $u(x):\mathbb{R}_{+}\to\mathbb{R}$, we construct $\tilde{u}(y,x):\mathbb{R}^{2}_{+}\to\mathbb{R}$ as
\begin{equation*}
    \tilde{u}(y,x):=u(x),\quad\mbox{for }y\in\mathbb{R},\ x\in\mathbb{R}_{+}.
\end{equation*}
It then follows that $\tilde{u}$ still satisfies the equation
\begin{equation*}
    (-\Delta)^s\tilde{u}(y,x)=\frac{x^{\alpha}}{\tilde{u}(y,x)^{\gamma}}.
\end{equation*}
We then apply Step 1-2 to $\tilde{u}(y,x)$, then the estimate for $u(x)$ follows immediately.
\end{itemize}
\end{proof}
\section{Classification of global solutions}
Now let's prove the remaining parts of Theorem~\ref{thm. classification}.
\subsection{The method of moving planes}\label{subsection: MMP}
In this subsection, when $\gamma$ satisfies \eqref{eq. existence requirement}. we show that the solution $u(x)$ is one-dimensional, meaning that it depends on $x_{n}$ only, thus proving Theorem~\ref{thm. classification} (1). Since $u$ does not tend to zero at infinity, we have to first convert $u$ into its Kelvin transform $v$, which tends to zero at infinity. Then we can apply the method of moving planes to obtain the one-dimensional symmetry of $u$.

\begin{proof}[Proof of Theorem~\ref{thm. classification} (1)]
It suffices to assume $n\geq2$. The key is to show that $u(x)$ is radial symmetric in the $x'$-direction with respect to any $x^{0}\in\partial\mathbb{R}^{n}_{+}$. Without loss of generality, we let $x^{0}=0$ and set
\begin{equation*}
    v(x)=\frac{1}{|x|^{n-2s}}u\left(\frac{x}{|x|^2}\right).
\end{equation*}
As we have $\|u\|_{L^{\infty}(B_{1}^{+})}<\infty$, it is readily apparent that
\begin{equation*}
\lim_{|x|\to\infty}v(x)=0
\end{equation*}
and it can be easily verified that for $x\in \mathbb{R}^{n}_{+}$,
\begin{equation*}
(-\Delta)^{s}v(x)=\frac{1}{|x|^\beta}\frac{x_{n}^{\alpha}}{v(x)^\gamma},\quad\mbox{with }\beta=n+2s+\gamma(n-2s)+2\alpha>0.
\end{equation*}

For any given $\lambda<0$, let $\tilde{T}_\lambda=\{x\in\mathbb{R}^{n}\mid x_1=\lambda \}$ be the moving planes and let
\begin{equation*}
\tilde{\Sigma}_{\lambda}=\{x\in\mathbb{R}^{n}\mid x_{1}<\lambda\},\quad\Sigma_{\lambda}=\tilde{\Sigma}_{\lambda}\cap\mathbb{R}^{n}_{+}.
\end{equation*}
Define $x^{\lambda}=(2\lambda-x_{1}, x_{2},\cdots,x_{n})$ as the reflection of $x$ about $\tilde T_\lambda$. Set $v_\lambda(x)=v(x^\lambda)$ and consider
\begin{equation*}
V_\lambda(x)=v_\lambda(x)-v(x).
\end{equation*}
From \eqref{eq. main}, we infer that
\begin{equation}\label{eq1-2}
(-\Delta)^{s}V_\lambda(x)=\frac{x_{n}^{\alpha}}{|x^{\lambda}|^\beta v(x^\lambda)^\gamma}-\frac{x_{n}^{\alpha}}{|x|^\beta v(x)^\gamma}.
\end{equation}
We also notice that $V_{\lambda}=0$ on $\tilde{T}_{\lambda}$, as $x=x^{\lambda}$ on that surface. In fact, we have $V_{\lambda}(x)=-V_{\lambda}(x^\lambda)$ for all $x$. Moreover, $v$ vanishes at the boundary $\partial\mathbb{R}^{n}$ continuously, except possibly at the origin. In conclusion, we obtain that the boundary value of $V_{\lambda}(x)$ in $\Sigma_{\lambda}$ satisfies:
\begin{equation}\label{eq. V lambda positive on the boundary}
    \liminf_{\substack{x\in\Sigma_{\lambda}\\x\to\partial\Sigma_{\lambda}\mbox{ or }|x|\to\infty}}V_{\lambda}(x)\geq0.
\end{equation}

To derive that $u(x)$ depends on $x_n$ only, we compare the value of $u$ at the point $x$ with the point $x^\lambda$. More precisely, we aim at showing that
\begin{equation}\label{eq1-3}
V_\lambda(x)\geq0,\quad\mbox{for all }x\in\Sigma_{\lambda}\mbox{ with }\lambda<0.
\end{equation}

In fact, suppose that \eqref{eq1-3} is false, then from \eqref{eq. V lambda positive on the boundary}, there exists a point $\bar{x}\in\Sigma_{\lambda}$ such that
\begin{equation*}
V_{\lambda}(\bar{x})=\min_{\Sigma_{\lambda}}V_{\lambda}<0.
\end{equation*}
From \eqref{eq1-2}, we see that at $\bar{x}$, if $v(\bar{x})>v_{\lambda}(\bar{x})=v(\bar{x}^\lambda)$, then
\begin{equation*}
(-\Delta)^s V_\lambda (\bar{x})=\frac{x_{n}^{\alpha}}{|\bar{x}^{\lambda}|^\beta v(\bar{x}^\lambda)^\gamma}-\frac{x_{n}^{\alpha}}{|\bar{x}|^\beta v(\bar{x})^{\gamma}}>0.
\end{equation*}
On the other hand, we claim that the ``derivative test of order $2s$" will imply
\begin{equation}\label{eq. second derivative test}
    (-\Delta)^{s}V_{\lambda}(\bar{x})\leq 0,
\end{equation}
and thus reaching a contradiction. To prove \eqref{eq. second derivative test}, we have
\begin{align*}
    &(-\Delta)^s V_{\lambda}(\bar x)=C_{n,s}PV\int_{\mathbb{R}^n}\frac{V_{\lambda}(\bar x)-V_{\lambda}(y)}{|\bar{x}-y|^{n+2s}}dy\\
    =&C_{n,s}PV\int_{\tilde\Sigma_{\lambda}}\frac{V_{\lambda}(\bar{x})-V_{\lambda}(y)}{|\bar{x}-y|^{n+2s}}dy+C_{n,s}\int_{\mathbb{R}^{n}\setminus\tilde\Sigma_{\lambda}}\frac{V_{\lambda}(\bar x)-V_{\lambda}(y)}{|\bar{x}-y|^{n+2s}}dy\\
    =&C_{n,s}PV\int_{\tilde\Sigma_{\lambda}}\frac{V_{\lambda}(\bar{x})-V_{\lambda}(y)}{|\bar{x}-y|^{n+2s}}dy+C_{n,s}\int_{\tilde\Sigma_{\lambda}}\frac{V_{\lambda}(\bar{x})+V_{\lambda}(y)}{|\bar{x}-y^{\lambda}|^{n+2s}}d y\\
    =&C_{n,s}\int_{\tilde\Sigma_{\lambda}}\frac{2V_{\lambda}(\bar{x})}{|\bar{x}-y^{\lambda}|^{n+2s}}dy+C_{n,s}PV\int_{\tilde\Sigma_{\lambda}}\Big(V_{\lambda}(\bar{x})-V_{\lambda}(y)\Big)\Big(\frac{1}{|\bar{x}-y|^{n+2s}}-\frac{1}{|\bar{x}-y^{\lambda}|^{n+2s}}\Big)dy\\
    \leq&C_{n,s}PV\int_{\Sigma_{\lambda}}\Big(V_{\lambda}(\bar{x})-V_{\lambda}(y)\Big)\Big(\frac{1}{|\bar{x}-y|^{n+2s}}-\frac{1}{|\bar{x}-y^{\lambda}|^{n+2s}}\Big)dy.
\end{align*}
Notice that in $\Sigma_{\lambda}$,
\begin{equation*}
    V_{\lambda}(\bar{x})-V_{\lambda}(y)\leq0,\mbox{ and }\Big(\frac{1}{|\bar{x}-y|^{n+2s}}-\frac{1}{|\bar{x}-y^{\lambda}|^{n+2s}}\Big)\geq0,
\end{equation*}
so we have shown \eqref{eq. second derivative test}. Therefore, \eqref{eq1-3} is indeed valid, which indicates that $v(x)$ is monotone increasing in the $x_{1}$-direction provided that $x_{1}<0$.

If we move the plane from the right side of $0$, then by a similar argument, we deduce that $v(x)$ is monotone increasing along the $x_{1}$-direction whenever $x_{1}>0$. As a consequence, $v(x)$ is symmetric about the plane $\tilde{T}_{0}$. Since the $x_{1}$-direction can be chosen arbitrarily, we have shown that $v(x)$ is radially symmetric about $0$ in the $x'$-direction. Moreover, by the translation invariance of \eqref{eq. main}, we conclude that $u(x)$ depends on $x_{n}$ only, i.e., $u(x)=u(x_{n})$. This finishes the proof of Theorem~\ref{thm. classification} (1).
\end{proof}

\subsection{Analysis of one-dimensional solutions}
Based on the established one-dimensional symmetry result, now \eqref{eq. main} can be transformed into a singular one-dimensional problem given by
\begin{equation}\label{eq. main1}
    \begin{cases}
    (-\Delta)_1^s u(t)=\frac{t^{\alpha}}{u(t)^\gamma},&t\in\mathbb{R}_{+},\\
    u(t)=0,& t\in\mathbb{R}\backslash\mathbb{R}_+.
    \end{cases}
\end{equation}
Now we study the classification of solutions to \eqref{eq. main1}. We first need two additional lemmas to further describe solutions $U_{K}(t)$'s mentioned in Theorem~\ref{thm. construction}.
\begin{lemma}[$U_{0}$ is the minimal solution]\label{lem. homogeneous solution is minimal}
If $u(t)$ is a global solution of \eqref{eq. main1}, then $u(t)$ is everywhere no less than $U_{0}(t)$ mentioned in Theorem~\ref{thm. construction} (1).
\end{lemma}
\begin{proof}
    We only show the minimality of $U_{0}$. We let $u_{b}(t)$ be the solution to
    \begin{equation}\label{eq. 1d in (0,b)}
        \left\{\begin{aligned}
            &(-\Delta)^{s}u_{b}(t)=\frac{t^{\alpha}}{u_{b}(t)^{\gamma}}&\mbox{in }(0,b)\\
            &u_{b}(t)=0&\mbox{elsewhere}
        \end{aligned}\right.,
    \end{equation}
    existence of which is guaranteed by Lemma~\ref{lem. well posedness}. We observe the following facts:
    \begin{itemize}
        \item $u_{b}(t)$ is increasing in $b$ by Lemma~\ref{lem. maximal principle};
        \item $u_{b}(t)\leq u(t)$ for every global solution $u(t)$, e.g., $U_{0}(t)$, because it has a smaller exterior condition;
        \item $u_{tb}(t)=t^{\frac{\alpha+2s}{\gamma+1}}u_{b}(1)$ by direct computation.
    \end{itemize}
    Then, we send $b\to\infty$, we see $u_{b}(t)\to u_{\infty}(t)\leq U_{0}(t)$ in the pointwise sense, and in fact, such a function $u_{\infty}(t)$ must be smaller than any global solution $u(t)$ to \eqref{eq. 1d in (0,b)} by Lemma~\ref{lem. maximal principle}.
    
    Now it suffices to show $u_{\infty}(t)=U_{0}$ (then, $U_{0}$ is the smallest global solution). It follows from Lemma~\ref{lem. monotone convergence} that $u_{\infty}(t)$ is a global solution to \eqref{eq. main1}, and we omit the details since it is similar to the computation in \eqref{eq. limit equation computation example}. Besides, by sending $b\to\infty$ in the relation $u_{tb}(t)=t^{\frac{\alpha+2s}{\gamma+1}}u_{b}(1)$, we have
    \begin{equation*}
        u_{\infty}(t)=t^{\frac{\alpha+2s}{\gamma+1}}u_{\infty}(1).
    \end{equation*}
    Consequently, $u_{\infty}(t)$ is also $\frac{\alpha+2s}{\gamma+1}$-homogeneous. By Lemma~\ref{lem. homogeneous solution derivative expression}, we have $u_{\infty}(t)=U_{0}$.
\end{proof}
\begin{lemma}[Continuity in $K$]\label{lem. continuity in slope}
    For every $K\geq0$, $U_{K}(x)$'s constructed in Theorem~\ref{thm. construction} are pointwise increasing and continuous in $K$. Precisely speaking, we have:
    \begin{equation*}
        \lim_{k\to K}U_{k}(t)=U_{K}(t),\quad\mbox{for all }t>0.
    \end{equation*}
\end{lemma}
\begin{proof}
    We let $K_{1}<K_{2}$ and consider
    \begin{equation*}
        v_{b}(t)=U_{K_{2},b}(t)-U_{K_{1},b}(t).
    \end{equation*}
    As the exterior condition of $U_{K_{2},b}(t)$ is larger than $U_{K_{1},b}(t)$, we have $U_{K_{2},b}(t)\geq U_{K_{1},b}(t)$ by Lemma~\ref{lem. maximal principle}. In other words, $v_{b}(t)\geq0$. By sending $b\to\infty$, we have $U_{K_{2}}(t)\geq U_{K_{1}}(t)$, proving the monotonicity in $K$. As for the continuity, we notice that the exterior condition for $v_{b}(t)$ is
    \begin{equation*}
        v_{b}(t)=(U_{0}(t)+K_{2}t^{s})_{+}-(U_{0}(t)+K_{1}t^{s})_{+}=(K_{2}-K_{1})(t^{s})_{+}.
    \end{equation*}
    The equation satisfied by $v_{b}(t)$ in $(0,b)$ is:
    \begin{equation*}
        (-\Delta)^{s}v_{b}=(-\Delta)^{s}U_{K_{2},b}(t)-(-\Delta)^{s}U_{K_{1},b}(t)=\frac{t^{\alpha}}{U_{K_{2},b}(t)^{\gamma}}-\frac{t^{\alpha}}{U_{K_{1},b}(t)^{\gamma}}\leq0.
    \end{equation*}
    Therefore, we have $0\leq v_{b}(t)\leq(K_{2}-K_{1})t^{s}$ in $(0,b)$. In conclusion, for any fixed $t$, as long as $K_{1}$ and $K_{2}$ are sufficiently close, $U_{K_{1},b}(t)$ and $U_{K_{2},b}(t)$ are also close to each other. Passing to the limit
    \begin{equation*}
        \lim_{|K_{2}-K_{1}|\to0}|U_{K_{2}}(t)-U_{K_{1}}(t)|\to0
    \end{equation*}
    in the pointwise sense, verifying the pointwise continuity in $K$.
\end{proof}
We are now able to prove Theorem~\ref{thm. classification} (2).
\begin{proof}[Proof of Theorem~\ref{thm. classification} (2)]
    It follows from Theorem~\ref{thm. classification} (1) that $u(x)$ depends only on $x_{n}$ when $n\geq2$, so we only need to consider the case $n=1$. Recall that from Lemma~\ref{lem. u is close to a linear function}, we can find some constant $K\geq0$ (depending on $u$), such that for every $t\in\mathbb{R}_{+}$,
    \begin{equation*}
        |u(t)-K t^{s}|\leq C_{4}t^{\frac{\alpha+2s}{\gamma+1}}.
    \end{equation*}
    From the construction of $U_{K}(t)$ in Theorem~\ref{thm. construction} (2), it should be expected that $u(t)=U_{K}(t)$. To see this, we first consider the case where $K>0$, and it suffices to show
    \begin{equation}\label{eq. assert what u is}
        U_{K-\varepsilon}(t)\leq u(t)\leq U_{K+\varepsilon}(t)\quad\mbox{for every }\varepsilon>0.
    \end{equation}
    Then, $u(t)=U_{K}(t)$ follows from Lemma~\ref{lem. continuity in slope}. To prove \eqref{eq. assert what u is}, we consider two cases:
    \begin{itemize}
        \item Case 1: When $K>0$, from $|u(t)-K t^{s}|\leq C_{4}t^{\frac{\alpha+2s}{\gamma+1}}$ and the asymptotic behavior of $U_{K}(t)$, we see
    \begin{equation*}
        u(t)\leq U_{K+\varepsilon}(t)\quad\mbox{in }(0,b)^{c},
    \end{equation*}
    for every $b\geq b_{0}(K,\varepsilon)$. Then by Lemma~\ref{lem. maximal principle}, we see
    \begin{equation*}
        u(t)\leq U_{K+\varepsilon}(t)\quad\mbox{in }(0,b)
    \end{equation*}
    as well. We then send $b\to\infty$, and prove the second inequality in \eqref{eq. assert what u is}. The first inequality in \eqref{eq. assert what u is} is argued similarly and we omit the proof.
    \item Case 2: When $K=0$, the second inequality of \eqref{eq. assert what u is} still holds true, meaning that at least we have $u(t)\leq U_{0}(t)$. Then from the minimality of $U_{0}$ (see Lemma~\ref{lem. homogeneous solution is minimal}), we have $u(t)=U_{0}(t)$.
    \end{itemize}
\end{proof}
\section{Further properties of global solutions}
\begin{proof}[Proof of Theorem~\ref{thm. further property} (1)]
    It suffices to combine \eqref{eq. U_K almost best estimate} with Lemma~\ref{lem. homogeneous solution is minimal}.
\end{proof}
\begin{proof}[Proof of Theorem~\ref{thm. further property} (2)]
    We notice by direct computation that
\begin{equation*}
    \lambda^{-\frac{\alpha+2s}{s(\gamma-1)-\alpha}}U_{K}(\lambda^{\frac{1+\gamma}{s(\gamma-1)-\alpha}}t)
\end{equation*}
is also a solution to \eqref{eq. main1}. Next, since we have
\begin{equation*}
    \lim_{t\to+\infty}\frac{\lambda^{-\frac{\alpha+2s}{s(\gamma-1)-\alpha}}U_{K}(\lambda^{\frac{1+\gamma}{s(\gamma-1)-\alpha}}t)}{t^s}=\lambda K,
\end{equation*}
then Theorem~\ref{thm. classification} (2) implies that $\lambda^{-\frac{\alpha+2s}{s(\gamma-1)-\alpha}}U_{K}(\lambda^{\frac{1+\gamma}{s(\gamma-1)-\alpha}}t)$ is nothing but $U_{\lambda K}$.
\end{proof}
\begin{proof}[Proof of Theorem~\ref{thm. further property} (3)]
    The monotonicity for $U_{0}(x_{n})$ is obvious, so we focus on showing the monotonicity of $U_{K}(x_{n})$ for $K>0$. The trick is to utilize the mutually scaling relation given in Theorem~\ref{thm. further property} (2).
    
    If \eqref{eq. existence requirement} holds, then we see
    \begin{equation*}
        -\frac{\alpha+2s}{s(\gamma-1)-\alpha}<0,\quad\frac{1+\gamma}{s(\gamma-1)-\alpha}>0.
    \end{equation*}
    For every $t_{1},t_{2}$ satisfying $0<t_{1}<t_{2}$, we denote
    \begin{equation*}
        \lambda=(\frac{t_{2}}{t_{1}})^{\frac{s(\gamma-1)-\alpha}{1+\gamma}}>1.
    \end{equation*}
    Then by Theorem~\ref{thm. further property} (2), we have
    \begin{equation*}
        U_{\lambda K}(t_{1})=\lambda^{-\frac{\alpha+2s}{s(\gamma-1)-\alpha}}U_{K}(\lambda^{\frac{1+\gamma}{s(\gamma-1)-\alpha}}t_{1})=\lambda^{-\frac{\alpha+2s}{s(\gamma-1)-\alpha}}U_{K}(t_{2})<U_{K}(t_{2}).
    \end{equation*}
    Notice that as $\lambda>1$, we can infer from Lemma~\ref{lem. continuity in slope} that $U_{\lambda K}(t_{1})>U_{K}(t_{1})$, so we have $U_{K}(t_{1})<U_{K}(t_{2})$, and hence we have finished the proof of Theorem~\ref{thm. further property} (3).
\end{proof}

\vspace{2mm}

\noindent \textbf{Acknowledgment.}
The authors  sincerely thank the editor and  the reviewers for their careful reading and insightful comments,  which have clarified ambiguities, corrected errors, and improved the clarity of this paper. 
Guo is partially supported by the NSFC-12501145,  the Natural Science Foundation of Shanghai (No. 25ZR1402207), the Postdoctoral Fellowship Program of CPSF (No. GZC20252004),  the China Postdoctoral Science Foundation (No. 2025T180838 and 2025M773061), and the Institute of Modern Analysis-A Frontier Research Center of Shanghai. Zhang is partially supported by NSFC-12526202, and NSFC-12141105.

\textbf{Conflict of interest.} The authors do not have any possible conflict of interest.

\vspace{2mm}

\textbf{Data availability statement.}
Data sharing is not applicable to this article as no data sets were generated or analyzed during the current study.

\bibliography{main.bib}

@misc {GuWu25,
    AUTHOR = {Yahong Guo and Leyun Wu},
     TITLE = {Classification of solutions to a singular fractional problem in the half space},
      YEAR = {2025, preprint},
}

@article {MMS1,
    AUTHOR = {Montoro, Luigi and Muglia, Luigi and Sciunzi, Berardino},
     TITLE = {The classification of all weak solutions to {$-\Delta
              U=U^{-\gamma}$} in the half space},
   JOURNAL = {SIAM J. Math. Anal.},
  FJOURNAL = {SIAM Journal on Mathematical Analysis},
    VOLUME = {57},
      YEAR = {2025},
    NUMBER = {5},
     PAGES = {5080--5088},
      ISSN = {0036-1410,1095-7154},
   MRCLASS = {35J75 (35A02 35B09)},
  MRNUMBER = {4959940},
MRREVIEWER = {Wen\ Yang},
       DOI = {10.1137/24M167189X},
       URL = {https://doi.org/10.1137/24M167189X},
}

@article {MMS,
    AUTHOR = {Montoro, Luigi and Muglia, Luigi and Sciunzi, Berardino},
     TITLE = {Classification of solutions to {$-\Delta u=u^{-\gamma}$} in
              the half-space},
   JOURNAL = {Math. Ann.},
  FJOURNAL = {Mathematische Annalen},
    VOLUME = {389},
      YEAR = {2024},
    NUMBER = {3},
     PAGES = {3163--3179},
      ISSN = {0025-5831,1432-1807},
   MRCLASS = {35J75 (35A02 35B09)},
  MRNUMBER = {4753083},
MRREVIEWER = {Long\ Wei},
       DOI = {10.1007/s00208-023-02717-4},
       URL = {https://doi.org/10.1007/s00208-023-02717-4},
}

@article {CES2019,
    AUTHOR = {Canino, Annamaria and Esposito, Francesco and Sciunzi,
              Berardino},
     TITLE = {On the {H}\"opf boundary lemma for singular semilinear
              elliptic equations},
   JOURNAL = {J. Differential Equations},
  FJOURNAL = {Journal of Differential Equations},
    VOLUME = {266},
      YEAR = {2019},
    NUMBER = {9},
     PAGES = {5488--5499},
      ISSN = {0022-0396,1090-2732},
   MRCLASS = {35J91 (35B09 35J25 35J75)},
  MRNUMBER = {3912757},
       DOI = {10.1016/j.jde.2018.10.039},
       URL = {https://doi.org/10.1016/j.jde.2018.10.039},
}

@article {JFA2020,
    AUTHOR = {Esposito, Francesco and Sciunzi, Berardino},
     TITLE = {On the {H}\"opf boundary lemma for quasilinear problems
              involving singular nonlinearities and applications},
   JOURNAL = {J. Funct. Anal.},
  FJOURNAL = {Journal of Functional Analysis},
    VOLUME = {278},
      YEAR = {2020},
    NUMBER = {4},
     PAGES = {108346, 25},
      ISSN = {0022-1236,1096-0783},
   MRCLASS = {35J92 (35J67 35J75)},
  MRNUMBER = {4044739},
       DOI = {10.1016/j.jfa.2019.108346},
       URL = {https://doi.org/10.1016/j.jfa.2019.108346},
}

@article {LM1991,
    AUTHOR = {Lazer, A. C. and McKenna, P. J.},
     TITLE = {On a singular nonlinear elliptic boundary-value problem},
   JOURNAL = {Proc. Amer. Math. Soc.},
  FJOURNAL = {Proceedings of the American Mathematical Society},
    VOLUME = {111},
      YEAR = {1991},
    NUMBER = {3},
     PAGES = {721--730},
      ISSN = {0002-9939,1088-6826},
   MRCLASS = {35J60 (35B65)},
  MRNUMBER = {1037213},
MRREVIEWER = {Michael\ Wiegner},
       DOI = {10.2307/2048410},
       URL = {https://doi.org/10.2307/2048410},
}

@article {delP1992,
    AUTHOR = {del Pino, Manuel A.},
     TITLE = {A global estimate for the gradient in a singular elliptic
              boundary value problem},
   JOURNAL = {Proc. Roy. Soc. Edinburgh Sect. A},
  FJOURNAL = {Proceedings of the Royal Society of Edinburgh. Section A.
              Mathematics},
    VOLUME = {122},
      YEAR = {1992},
    NUMBER = {3-4},
     PAGES = {341--352},
      ISSN = {0308-2105,1473-7124},
   MRCLASS = {35J65 (35B65)},
  MRNUMBER = {1200204},
MRREVIEWER = {Xiaoyun\ Ma},
       DOI = {10.1017/S0308210500021144},
       URL = {https://doi.org/10.1017/S0308210500021144},
}

@article {CRT1977,
    AUTHOR = {Crandall, M. G. and Rabinowitz, P. H. and Tartar, L.},
     TITLE = {On a {D}irichlet problem with a singular nonlinearity},
   JOURNAL = {Comm. Partial Differential Equations},
  FJOURNAL = {Communications in Partial Differential Equations},
    VOLUME = {2},
      YEAR = {1977},
    NUMBER = {2},
     PAGES = {193--222},
      ISSN = {0360-5302,1532-4133},
   MRCLASS = {35J65},
  MRNUMBER = {427826},
MRREVIEWER = {Michael\ Wiegner},
       DOI = {10.1080/03605307708820029},
       URL = {https://doi.org/10.1080/03605307708820029},
}

@article {OP2018,
    AUTHOR = {Oliva, Francescantonio and Petitta, Francesco},
     TITLE = {Finite and infinite energy solutions of singular elliptic
              problems: existence and uniqueness},
   JOURNAL = {J. Differential Equations},
  FJOURNAL = {Journal of Differential Equations},
    VOLUME = {264},
      YEAR = {2018},
    NUMBER = {1},
     PAGES = {311--340},
      ISSN = {0022-0396,1090-2732},
   MRCLASS = {35J61 (35J75 35J91 35R06)},
  MRNUMBER = {3712944},
MRREVIEWER = {Alan\ V.\ Lair},
       DOI = {10.1016/j.jde.2017.09.008},
       URL = {https://doi.org/10.1016/j.jde.2017.09.008},
}

@article {BO2010,
    AUTHOR = {Boccardo, Lucio and Orsina, Luigi},
     TITLE = {Semilinear elliptic equations with singular nonlinearities},
   JOURNAL = {Calc. Var. Partial Differential Equations},
  FJOURNAL = {Calculus of Variations and Partial Differential Equations},
    VOLUME = {37},
      YEAR = {2010},
    NUMBER = {3-4},
     PAGES = {363--380},
      ISSN = {0944-2669,1432-0835},
   MRCLASS = {35J25 (35J67)},
  MRNUMBER = {2592976},
MRREVIEWER = {Cristina\ Trombetti},
       DOI = {10.1007/s00526-009-0266-x},
       URL = {https://doi.org/10.1007/s00526-009-0266-x},
}

@article {AGJ2018,
    AUTHOR = {Adimurthi and Giacomoni, Jacques and Santra, Sanjiban},
     TITLE = {Positive solutions to a fractional equation with singular
              nonlinearity},
   JOURNAL = {J. Differential Equations},
  FJOURNAL = {Journal of Differential Equations},
    VOLUME = {265},
      YEAR = {2018},
    NUMBER = {4},
     PAGES = {1191--1226},
      ISSN = {0022-0396,1090-2732},
   MRCLASS = {35R11 (35B09 35B65)},
  MRNUMBER = {3797614},
       DOI = {10.1016/j.jde.2018.03.023},
       URL = {https://doi.org/10.1016/j.jde.2018.03.023},
}

@article {CMSS2017,
    AUTHOR = {Canino, Annamaria and Montoro, Luigi and Sciunzi, Berardino
              and Squassina, Marco},
     TITLE = {Nonlocal problems with singular nonlinearity},
   JOURNAL = {Bull. Sci. Math.},
  FJOURNAL = {Bulletin des Sciences Math\'ematiques},
    VOLUME = {141},
      YEAR = {2017},
    NUMBER = {3},
     PAGES = {223--250},
      ISSN = {0007-4497,1952-4773},
   MRCLASS = {35R11 (35A15 35J62)},
  MRNUMBER = {3639996},
MRREVIEWER = {Patrizia\ Pucci},
       DOI = {10.1016/j.bulsci.2017.01.002},
       URL = {https://doi.org/10.1016/j.bulsci.2017.01.002},
}

@article {GMS2017,
    AUTHOR = {Giacomoni, Jacques and Mukherjee, Tuhina and Sreenadh,
              Konijeti},
     TITLE = {Positive solutions of fractional elliptic equation with
              critical and singular nonlinearity},
   JOURNAL = {Adv. Nonlinear Anal.},
  FJOURNAL = {Advances in Nonlinear Analysis},
    VOLUME = {6},
      YEAR = {2017},
    NUMBER = {3},
     PAGES = {327--354},
      ISSN = {2191-9496,2191-950X},
   MRCLASS = {35R11 (35A15 35R09)},
  MRNUMBER = {3680366},
       DOI = {10.1515/anona-2016-0113},
       URL = {https://doi.org/10.1515/anona-2016-0113},
}

@article {BCT2014,
    AUTHOR = {Brandolini, B. and Chiacchio, F. and Trombetti, C.},
     TITLE = {Symmetrization for singular semilinear elliptic equations},
   JOURNAL = {Ann. Mat. Pura Appl. (4)},
  FJOURNAL = {Annali di Matematica Pura ed Applicata. Series IV},
    VOLUME = {193},
      YEAR = {2014},
    NUMBER = {2},
     PAGES = {389--404},
      ISSN = {0373-3114,1618-1891},
   MRCLASS = {35J91 (35B45 35B50 35B51)},
  MRNUMBER = {3180924},
       DOI = {10.1007/s10231-012-0280-z},
       URL = {https://doi.org/10.1007/s10231-012-0280-z},
}

@book {CLM2020,
    AUTHOR = {Chen, Wenxiong and Li, Yan and Ma, Pei},
     TITLE = {The fractional {L}aplacian},
 PUBLISHER = {World Scientific Publishing Co. Pte. Ltd., Hackensack, NJ},
      YEAR = {[2020] \copyright 2020},
     PAGES = {x+331},
      ISBN = {[9789813223998]; [9789813224001]; [9789813224018]},
   MRCLASS = {35R11},
  MRNUMBER = {4274583},
       DOI = {10.1142/10550},
       URL = {https://doi.org/10.1142/10550},
}

@article {FuMa,
    AUTHOR = {Fulks, W. and Maybee, J. S.},
     TITLE = {A singular non-linear equation},
   JOURNAL = {Osaka Math. J.},
  FJOURNAL = {Osaka Mathematical Journal},
    VOLUME = {12},
      YEAR = {1960},
     PAGES = {1--19},
      ISSN = {0388-0699},
   MRCLASS = {35.62},
  MRNUMBER = {123095},
MRREVIEWER = {D.\ G.\ Aronson},
}

@article {No,
    AUTHOR = {Nowosad, P.},
     TITLE = {On the integral equation {$\kappa f=1/f$} arising in a problem
              in communication},
   JOURNAL = {J. Math. Anal. Appl.},
  FJOURNAL = {Journal of Mathematical Analysis and Applications},
    VOLUME = {14},
      YEAR = {1966},
     PAGES = {484--492},
      ISSN = {0022-247X},
   MRCLASS = {94.10},
  MRNUMBER = {195634},
MRREVIEWER = {J.\ Capon},
       DOI = {10.1016/0022-247X(66)90008-4},
       URL = {https://doi.org/10.1016/0022-247X(66)90008-4},
}

@article {Fo,
    AUTHOR = {Fowler,R.H. },
     TITLE = {The solution of {E}mden's and similar differential equations},
   JOURNAL = {Monthly Notices Roy. Astro. Soc.},
  FJOURNAL = {Journal of Mathematical Analysis and Applications},
    VOLUME = {91},
      YEAR = {1930},
     PAGES = {63--91},
}

@article {Pe,
    AUTHOR = {Perry, W.L.},
     TITLE = {A monotone iterative technique for solution of {$p$}th order
              {$(p<0)$} reaction-diffusion problems in permeable
              catalysis},
   JOURNAL = {J. Comput. Chem.},
  FJOURNAL = {Journal of Computational Chemistry},
    VOLUME = {5},
      YEAR = {1984},
    NUMBER = {4},
     PAGES = {353--357},
      ISSN = {0192-8651},
   MRCLASS = {80A32 (65L10)},
  MRNUMBER = {760372},
MRREVIEWER = {Renuka Datta},
       DOI = {10.1002/jcc.540050412},
       URL = {https://doi.org/10.1002/jcc.540050412},
}

@article {AcShPe,
    AUTHOR = {Acrivos, A. and Shah, M. J. and Petersen, E. E.},
     TITLE = {On the flow of non-{N}ewtonian liquid on a rotating disk},
   JOURNAL = {J. Appl. Phys.},
  FJOURNAL = {Journal of Applied Physics},
    VOLUME = {31},
      YEAR = {1960},
     PAGES = {963--968},
      ISSN = {0021-8979},
   MRCLASS = {76.00 (73.00)},
  MRNUMBER = {112518},
MRREVIEWER = {J. E. Adkins},
}

@article {NaCa,
    AUTHOR = {Nachman, A. and Callegari, A.},
     TITLE = {A nonlinear singular boundary value problem in the theory of
              pseudoplastic fluids},
   JOURNAL = {SIAM J. Appl. Math.},
  FJOURNAL = {SIAM Journal on Applied Mathematics},
    VOLUME = {38},
      YEAR = {1980},
    NUMBER = {2},
     PAGES = {275--281},
      ISSN = {0036-1399},
   MRCLASS = {76A05 (35Q20 76D10)},
  MRNUMBER = {564014},
MRREVIEWER = {Jacques Mauss},
       DOI = {10.1137/0138024},
       URL = {https://doi.org/10.1137/0138024},
}

@article {VaSoMo,
    AUTHOR = {Vajravelu, K. and Soewono, E. and Mohapatra, R. N.},
     TITLE = {On solutions of some singular, nonlinear differential
              equations arising in boundary layer theory},
   JOURNAL = {J. Math. Anal. Appl.},
  FJOURNAL = {Journal of Mathematical Analysis and Applications},
    VOLUME = {155},
      YEAR = {1991},
    NUMBER = {2},
     PAGES = {499--512},
      ISSN = {0022-247X},
   MRCLASS = {34B15 (35Q30 76D05)},
  MRNUMBER = {1097296},
       DOI = {10.1016/0022-247X(91)90015-R},
       URL = {https://doi.org/10.1016/0022-247X(91)90015-R},
}
\bibliographystyle{abbrv}

\end{document}